\documentclass{amsart}
\usepackage{amsmath,amsthm,amsopn,amsfonts,amssymb}

\newtheorem{thm}{Theorem}[section]

\newtheorem{lem}[thm]{Lemma}
\newtheorem{lemma}[thm]{Lemme}
\newtheorem{prop}[thm]{Proposition}
\theoremstyle{definition}
\newtheorem{defn}[thm]{Definition}
\theoremstyle{remark}
\newtheorem{rem}[thm]{Remark}

\newtheorem{example}[thm]{Example}
\numberwithin{equation}{section}

\newcommand{\norm}[1]{\left\Vert#1\right\Vert}

\newcommand{\eps}{\varepsilon}

\newcommand{\E}{\mathcal{E}}
\newcommand{\F}{\mathcal{F}}
\newcommand{\Haus}{\mathcal{H}}
\newcommand{\Leb}{\mathcal{L}}
\newcommand{\Saut}{J}
\newcommand{\dl}{d_{l^1}}

\newcommand{\W}{\mathrm{W}}
\newcommand{\N}{\mathbb{N}}
\newcommand{\R}{\mathbb{R}}

\newcommand{\barom}{\overline{\Omega}}
\newcommand{\om}{\Omega}
\newcommand{\dom}{{\partial \Omega}}

\def\build#1_#2^#3{\mathrel{\mathop{\kern 0pt#1}\limits_{#2}^{#3}}}

\begin{document}
\title{A particular class of solutions of a system of eikonal equations}

\author{Thierry Champion and Gisella Croce}
\address{D\'epartement de Math\'ematiques Universit\'e du Sud Toulon-Var, 
D\'epar\-te\-ment de Math\'e\-ma\-ti\-ques Universit\'e de Montpellier II}
\email{champion@univ-tln.fr, croce@math.univ-montp2.fr}
\thanks{The research of Gisella Croce was partially supported by the {\it Fonds National Suisse}.}
\subjclass{28A75, 34A60, 35F30, 35A15, 49J45}
\keywords{almost everywhere solutions, eikonal equation, 
functions of bounded variation, direct methods of the calculus of variations}
\begin{abstract}
In this paper we study the following Dirichlet problem:
\[
\left\{
\begin{array}{ll}
\displaystyle \left|\frac{\partial u}{\partial x_i}\right|=1\,\,
&\mbox{a.e. in} \,\Omega,\, i=1,\ldots,N
\\
u=0 & \mbox{on} \,\,\, \partial \Omega,
\end{array}
\right. 
\]
where $\Omega$ is an open bounded subset of $\R^N$ and 
$u \in W^{1,\infty}(\Omega)$. 
This problem
has infinitely many solutions: our aim is to isolate 
those which minimize some
functional involving the discontinuity
sets of $\frac{\partial u}{\partial x_i}$, for $i=1,\ldots,N$.
\end{abstract}

\maketitle


\section{Introduction}

In this paper we are interested in the following system of eikonal equations
\begin{equation}
\left\{
\begin{array}{ll}
\displaystyle \left|\frac{\partial u}{\partial x_i}\right|=1\,\,
&\mbox{a.e. in} \,\Omega,\, i=1,\ldots,N
\vspace{0.1cm}\\
u=0 & \mbox{on} \,\,\, \partial \Omega,
\end{array}
\right. \label{pbdipart}
\end{equation}
where $\om$ is a connected open bounded subset of $\R^N$ with Lipschitz boundary
and $u \in \W^{1,\infty}(\om) \cap C_0(\om)$.
This problem admits infinitely many solutions 
(see \cite{Cellina} and \cite{implicit} for example for a proof). 
An interesting question is then to isolate a particular class of solutions
which 
have some additional properties, 
since such an analysis could be useful to select 
and characterize a prefered solution.

There exist some evident difficulties to this kind of question. 
For example the lack of convexity of the set $\mathfrak{S}(\Omega)$ 
of the solutions 
of (\ref{pbdipart}) implies that the natural functionals 
$$
v \to \int_{\Omega}|v|^p\,\quad p\geq 1, 
$$
have in general neither  a minimizer nor a maximizer 
over $\mathfrak{S}(\Omega)$. 
Indeed any minimizing sequence converges to 0, which does not belong to $\mathfrak{S}(\Omega),$ 
and the cluster points of maximizing sequences are the distance function to $\partial \Omega$ 
(in the $l^1$ norm of $\R^N$) or its opposite, which usually do 
not verify (\ref{pbdipart}) (see remark \ref{gradient_distance}).

Another natural way to select a particular solution in $\mathfrak{S}(\Omega)$
would be to use the notion of viscosity solution 
(for the definition and further details see \cite{implicit}). 
Indeed one could study the existence of 
$W^{1,\infty}(\Omega)$ viscosity solutions of a problem like
\begin{equation}
\left\{
\begin{array}{ll}
\displaystyle F(Du)=0\,\,
&\mbox{a.e. in} \,\Omega
\vspace{0.1cm}\\
u=0 &\mbox{on}\, \partial \Omega
\end{array}
\right.\label{viscointroduction}
\end{equation}
where $F:\R^N\to \R$ is such that $F(x)=0$ if and only if $x=(x_1,\ldots,x_N)$ verifies $|x_i|=1\,\,\, \forall\,i=1,\ldots,x_N$.
This is indeed possible in dimension $N=1$, where the only viscosity solution of the problem 
\begin{equation}
\left\{
\begin{array}{ll}
\displaystyle \left|u'\right|-1=0\,\,
&\mbox{a.e. in} \,]-1,1[
\vspace{0.1cm}\\
u(-1)=u(1)=0&
\end{array}
\right.\label{pbdim1}
\end{equation}
 is $u(x)=1-|x|$.
Nevertheless when $N\geq 2$
there is in general no viscosity solution of (\ref{viscointroduction}), 
unless $\Omega$ is convex and the normals to $\partial \Omega$ are in the set $E=\{x: |x_i|=1/\sqrt{N}\};$   
we refer to \cite{CDGG} and \cite{Giovanni} for further details.

In the literature we find few articles 
with the purpose of 
isolating particular solutions of problem (\ref{pbdipart}) and 
selecting one. 
In \cite{pan} and \cite{daco-glow} the authors study 
numerically a variational problem over the set of non-negative 
solutions of (\ref{pbdipart}) in the 2-dimensional case: 
they obtain a maximizing sequence for the problem
\begin{equation}\label{glow}
\sup\left\{\int_\Omega u, \quad u\geq 0,\,\, u \in \mathfrak{S}(\Omega) 
\right\}
\end{equation}
through the numerical minimization of the functional
$$
u \mapsto -\int_\Omega u + \frac{1}{2}\int | \nabla u |^2 
+\frac{\eps}{2}\int |\Delta u|^2 + \frac{1}{2\eps}
\sum_{i=1}^{N} \int\left(\left|\frac{\partial u}{\partial x_i}\right| -1 \right)^2\,.
$$
Unfortunately, as we said above, there is in general no optimal solution 
for the variational problem (\ref{glow}), since the limit for any maximizing
sequence is the distance to $\partial \Omega$ (in the $l^1$ norm).

A different approach is followed in \cite{daco-marc-visc} where
the authors
construct an explicit solution. Indeed they
 define a special
partition of $\Omega$ made up of convex sets whose normals to $\partial \Omega$
 are in $E$; over each of these sets
they consider the relative viscosity solution. 
In this way the constructed solution has local properties (which depend on the
chosen partition), but no global ones. 
In any case we want to point out that this study is made in a more general 
framework than problem (\ref{pbdipart}).

Our purpose in this work is to characterize a particular 
class of solutions of 
(\ref{pbdipart}) through a variational method, 
i.e., as the optimal solutions of a variational problem. We
shall define over 
$\mathfrak{S}(\Omega)$ a functional involving
 the discontinuity sets of the gradient of the solutions of 
(\ref{pbdipart}) and we minimize it.
Let us briefly explain the kind of results that we shall present 
in this article. 
We distinguish two cases, according to the properties of $\Omega$. 
In the case where $\partial \Omega$ is composed by a finite set of faces
with normals in $E$ (see section \ref{sec-part-case}), we 
consider the functional
$$
v \,\,\mapsto\,\, \sum^N_{i=1} \Haus^{N-1}({J}_{\frac{\partial v}{\partial x_i}})
$$
(where ${J}_{\frac{\partial v}{\partial x_i}}$ is the 
jump set of $\frac{\partial v}{\partial x_i}$) and
 we show that it admits a minimizer.

In the general case, that is, when $\Omega$ is any connected open subset of $\R^N$ with Lipschitz boundary, 
we cannot minimize the discontinuity sets of $Du$ for $u \in \mathfrak{S}(\Omega)$
in the same way. Indeed the previous functional is in general infinite over $\mathfrak{S}(\Omega)$, as shown in section \ref{sec-gen-case}.
Consequently 
we construct a positive ``weight'' function $h \in C_0(\Omega)$ such that
the functional
$$
v \to\sum^N_{i=1}\int_{{J}_{\frac{\partial v}{\partial x_i}}}
h(x) d\mathcal{H}^{N-1}(x)
$$
admits  a minimizer over $\mathfrak{S}(\Omega)$. 

We also study the relationship between the viscosity solution (when it exists) 
and the optimal solutions of the minimization problem: we show that they coincide if $N=1,2$
(see Proposition \ref{viscosity_distance}).

Our study allows us to isolate particular solutions 
of problem (\ref{pbdipart}), but
unfortunately we are not able to say if it selects a particular one,
in the sense that the variational problems we propose could admit more than one solution.
In section \ref{sec_open_problems} we then present a method to select a function
among the optimal solutions of the variational problems that we studied (see {\bf 4.}).

Finally we refer to a recent work \cite{bellettini} which does not address the problem
of the selection of a particular solution of (\ref{pbdipart}) but is also concerned
with some local properties of the discontinuity sets of $Du, u \in \mathfrak{S}(\Omega)$.


\section{Notations}

Throughout this paper $\Omega$ denotes a bounded connected open subset
of $\R^N$ with Lipschitz boundary. 
We will denote by $(x_i)_{1 \leq i \leq N}$ the coordinates of a point $x \in \R^N$, and
$\langle x,y \rangle$ the usual scalar product of $x,y \in \R^N$. 
The norm $\norm{x}_{l^1}$ of $x\in \R^N$ stands for the usual $l^1$ norm,
that is $\norm{x}_{l^1}=\sum_{i=1}^N |x_i|\,.$

In the following, we shall use the short notation $W^{1,\infty}_0(\om)$ for
$W^{1,\infty}(\om) \cap C_0(\om)$. When considering a function $v \in W_0^{1,\infty}(\om)$, we shall
always assume that we are handling its continuous representative; 
notice that, due to the hypothesis on $\Omega,$ 
such a function $v$ is a Lipschitz continuous function.

Let us briefly recall some properties and notations of the spaces $BV$ and $SBV$ for future use 
(in any case we refer to \cite{ambrosio_fusco_pallara} and to \cite{evans-gariepy} for further details).
If $w \in BV(\Omega)$, we write $\norm{w}_{BV(\Omega)}$ 
for the usual $BV$ norm of $w$ defined by
$$
\norm{w}_{BV(\Omega)}=\norm{w}_{L^1(\Omega)}+|Dw|(\Omega)\,.
$$
Moreover we will use the usual weak* convergence in $BV$ defined as follows (see \cite{ambrosio_fusco_pallara}):
\begin{defn}[weak* convergence in $BV$]
Let $(u_n)_n, u \in BV(\Omega)$. We say that $(u_n)_n$ weakly* converges to $u$ in $BV(\Omega)$ if
$u_n \to u$ in $L^1(\Omega)$ and the measures $Du_n$ weakly* converge to the measure $Du$ in
$\mathcal{M}(\Omega, \R^N)$, that is
$$
\lim\limits_{n \to \infty}\int_{\Omega}\varphi dD u_n=
\int_{\Omega}\varphi dD u \qquad \forall\,\varphi \in C_0(\Omega)\,.
$$
\end{defn}
A $BV$ function belongs to $SBV$ when the Cantor part of its derivative is zero.
This implies that the measure $Dw$ of an $SBV(\Omega)$ function 
can be decomposed in
the following way:
\begin{eqnarray*}
Dw & = & D^a w + D^j w \\
& = & \nabla w \, \mathcal{L}^N+(w^+-w^-)\nu_w\,\mathcal{H}^{N-1}\lfloor \Saut_w
\,\, = \,\,
\nabla w\mathcal{L}^N+[w]\nu_w\mathcal{H}^{N-1}\lfloor \Saut_w
\end{eqnarray*}
where $D^a w$ is the absolute continuous part of $D w$ with respect to
the Lebesgue measure $\Leb^N$ in $\R^N$, $D^j w = [w]\nu_w\mathcal{H}^{N-1}\lfloor \Saut_w$
is the {\it jump} part of $Dw$,
$\Haus^{N-1}$ the Hausdorff measure of dimension $N-1$,
$w^+$ and $w^-$ denote the upper and lower approximate limits of $w$,
$J_w$ the jump set of $w$ and $\nu_w$ its generalized normal.

Referring to our problem, we will denote by
$\mathfrak{S}(\Omega)$ the set of $W^{1,\infty}(\Omega)$ solutions of 
(\ref{pbdipart}), that is
$$
\mathfrak{S}(\Omega):=\left\{ v \in W^{1,\infty}_0(\om) \, : \,\left|\frac{\partial v}{\partial x_i}\right|=1
\,\, a.e. \,\,in\,\, \om, i=1,\ldots,N \right\}\,.
$$
In this paper we will consider the piecewise $C^1$ elements 
$v$ of $\mathfrak{S}(\Omega)$ and more generally those elements $v$
of $\mathfrak{S}(\Omega)$ such that $Dv \in SBV_{loc}(\Omega)^N$.
To define a  piecewise $C^1$ function, 
we first recall the definition of Caccioppoli partitions and 
piecewise constant function (we refer to \cite{ambrosio_fusco_pallara} for further details):
\begin{defn}[Caccioppoli partition]
Let $I\subset \N$. A partition  $\{E_i\}_{i\in I}$ 
of $\Omega$ is a Caccioppoli partition if
$\displaystyle \sum\limits_{i\in I}|D\chi_{E_i}|(\Omega)<\infty$.
\end{defn}
In the above definition, $\chi_A$ denotes the characteristic function of a set $A$, that is
$\chi_A(x)=1$ when $x \in A$, $\chi_A(x)=0$ when $x \notin A$.
\begin{defn}[Piecewise constant function]
\label{piecewise_constant_maps_definition}
We say that $u: \Omega \to \R$ is piecewise constant in 
$\Omega$ if there exist a Caccioppoli partition $\{E_i\}_{i\in I}$ 
of $\Omega$  
and a function $t: I\to \R$ 
such that $\displaystyle u=\sum\limits_{i\in I}t_i\chi_{E_i}$.
\end{defn}
Notice that a piecewise constant function is an $SBV$ function (see \cite{ambrosio_fusco_pallara} for further details).

We  define the set of piecewise $C^1$ functions by
$$
C^1_{pw}(\Omega):=\left\{v \in W^{1,\infty}(\om): \frac{\partial v}{\partial x_i}\, \textnormal{is piecewise constant in $\Omega$}, 
i=1,\ldots,N\right\}\,.
$$

As has  already been said, we shall handle the solutions of (\ref{pbdipart}) such that $Dv \in SBV_{loc}(\Omega)^N$.
We now want to write their distributional gradient: for any index $i \in \{1,\ldots,N\}$ we  use the short notation
$\Saut^{v}_i$ instead of 
$\Saut_{\frac{\partial v}{\partial x_i}}$.
Therefore
for a function 
$v\in \mathfrak{S}(\Omega)$ such that $Dv \in SBV_{loc}(\Omega)^N$ one has
\begin{equation}\label{DVrepresentation}
D\left(\frac{\partial v}{\partial x_i}\right)
\,\, = \,\, 2 \, \nu_{\frac{\partial v}{\partial x_i}} \,\Haus^{N-1} \lfloor \Saut^v_i
\quad on \,\, \omega\quad \forall\, i=1,\ldots,N\,
\end{equation}
for any open subset $\omega \subset \overline{\omega}\subset \om$.

In the rest of the paper, the distance in $l^1$ norm to a  subset of $\R^N$ will
play an important role: if $A$ is an open bounded subset of $\R^N$, we define
$\dl(\cdot\,,A): \R^N \to \R$ by 
$$
\forall x \in \R^N, \quad \quad \quad
d_{l^1}(x,A) \,\, := \,\, \min\left\{ \|x-y\|_{l^1} \, : \, y \in A \right\}\,.
$$

\section{A particular case} \label{sec-part-case}
In this section we will assume that $\Omega$
satisfies the following property:
$$
(H) \quad \quad
\left\{\begin{array}{l}
\mbox{the boundary $\dom$ is the union of a finite family of polyhedrons} \\
\mbox{of dimension $N-1$, each of which is included in a hyperplane}\\
\mbox{with normals in $E$,}
\end{array}\right.
$$
where $E := \{ x \in \R^N\,:\, \forall i,\,\, |x_i|={1}/{\sqrt{N}} \}$, and we will minimize the $\mathcal{H}^{N-1}$ measure of the discontinuity sets of $\frac{\partial v}{\partial x_i},$ for $v \in \mathfrak{S}(\Omega).$
We  define
$$
\mathcal{E}_{pw}(\Omega):=
\{v \in \mathfrak{S}(\Omega):  v \in C^1_{pw}(\Omega)\}\,.
$$
We now state the main result of this section.
\begin{thm}\label{existence_particular_case}
Let $\Omega$ be an open bounded connected subset of $\R^N$ which satisfies hypothesis $(H)$. Then the variational problem
\[
(P_1) \quad \quad \quad \inf\left\{\,\F(v) \,:= \, \sum_{i=1}^N \Haus^{N-1}(\Saut_i^v)
\, : \, v \in \E_{pw}(\om) \right\}.
\]
has finite value and admits an optimal solution.
\end{thm}
\begin{rem}
Under assumption $(H)$, problem $(P_1)$ always has at least two optimal solutions:
indeed, if $u$ is a solution then $-u$ is also a solution.
\end{rem}
\begin{rem}
When $\om$ fails to satisfy hypothesis $(H)$, then problem $(P_1)$
is not well posed as we will show in section \ref{sec-gen-case}.
\end{rem}
In the proof of Theorem \ref{existence_particular_case} we will use the following lemmas.
\begin{lem}\label{Ebound}
Let $\Omega$ be an open bounded connected subset of $\R^N$ with Lipschitz boundary. Then
$$
- \dl(\cdot\, , \dom) \,\, \leq \,\, v \,\, \leq \,\, \dl(\cdot\, , \dom)
\quad on\,\, \overline{\Omega}
$$
for every function $v \in \mathfrak{S}(\Omega)$.
\end{lem}

\begin{proof}
Let $v\in \mathfrak{S}(\Omega).$
We denote by $A$ a set of zero Lebesgue measure such that
$v$ is differentiable on $\om \setminus A$ with
$\left|\frac{\partial v}{\partial x_i}\right|=1$ for all $i$.

Take $x \in \om$, and let $y \in \partial \om$ be such that
$d_{l^1}(x , \partial \om)=\norm{x-y}_{l^1}$.
For any $t \in \,]0,1]$ we define the map $\gamma_t: [0,1] \to \om$
by $\gamma_t(s):= s x + (1-s) x^t$, where $x^t := t x + (1-t) y$.
For a fixed $t \in \,]0,1[\,$ and $\eps >0$, we apply Lemma 3.2 in \cite{Briani} to get
the existence of a Lipschitz continuous path $\gamma_{t,\eps}: [0,1] \to \om \setminus A$
such that
$$
\gamma_{t,\eps}(0)=x^t, \,\, \gamma_{t,\eps}(1)=x \,\,\, and \,\,\,
\| \gamma_{t,\eps}' - \gamma_t' \|_{L^\infty(0,1)} \leq \eps.
$$
Then the function $v \circ \gamma_{t,\eps}$ is Lipschitz continuous on $[0,1]$
and we can compute
$$
v(x) - v(x^t)  =  \int_0^1 (v \circ \gamma_{t,\eps})'(s) ds
=\int_0^1 \nabla v (\gamma_{t,\eps}(s)).\gamma_{t,\eps}'(s) ds\,.
$$
Now $\left|\frac{\partial v}{\partial x_i}(\gamma_{t,s}(s))\right|=1$ for all $i$ and so
$$
v(x)-v(x^t)
\leq \int_0^1 \left\| \gamma_{t,\eps}'(s) \right\|_{l^1} ds
\leq  N \eps + \int_0^1 \left\| \gamma_t'(s) \right\|_{l^1} ds
\,\,\, = \,\,\, N\eps + \left\| x-x^t \right\|_{l^1}.
$$
Letting $t$ and $\eps$ go to $0$ this yields $v(x) \leq d_{l^1}(x , \partial \om)$.
In an analogous way one can prove that $-v(x) \leq d_{l^1}(x , \partial \om)$
\end{proof}

\begin{lem}\label{distanceH}
Let $\Omega$ be an open bounded connected subset of $\R^N$ which satisfies hypothesis $(H)$. Then $\dl(\cdot\,,\dom))$
belongs to $\E_{pw}(\Omega)$ and $\F(\dl(\cdot\,,\dom)) < +\infty$.
\end{lem}

\begin{rem}\label{gradient_distance} Notice that for a general bounded open subset $\om$, $d_{l^1}(\,\cdot\, , \partial \om)$
does not necessarily satisfy
$\left|\frac{\partial v}{\partial x_i}\right|=1$ a.e. in $\Omega$ for all $i$.
Indeed, it is sufficient to consider $\Omega=(0,1)^2\subset \R^2:$ we have that
$d_{l^1}((x_1,x_2) , \partial \om) = x_2$
in the set $T=\{(x_1,x_2):0\leq x_1 \leq 1/2,\,0\leq x_2 \leq x_1\}\,.$
\end{rem}
\begin{proof}
Since $\om$ satisfies hypothesis $(H)$, we can write
$\displaystyle \dom \,\, = \,\, \bigcup_{j=1}^{m} F_j$ where each face $F_j$
is a polyhedron of dimension $N-1$ included in a hyperplane $H_j$ with normal in $E$.
We shall assume that two faces $F_j$ and $F_k$ that intersect are not included
in the same hyperplane of $\R^N$.

Let $j \in \{1,\ldots,m\}$. We first notice that since the normal of the hyperplane
$H_j$ is included in $E$, one has
$\left|\frac{\partial }{\partial x_i} d_{l^1}(\,\cdot\,,H_j) \right|=1$ a.e. in $\Omega$
for all $i$. Let us prove that
$\displaystyle \dl(\,\cdot\,,\dom) = \min_{1 \leq j \leq m} \dl(\cdot\,,F_j)$ also satisfies
this system of eikonal equations in $\Omega$.
We fix $x$ in $\om$ and distinguish three cases:

{\em Case 1.\,} Assume that $\displaystyle \dl(x,\dom) = \dl(x,F_j)$ for a unique
index $j$. Then $\dl(x,F_j)$ is attained in the relative interior of $F_j$,
so that in a neighborhood of $x$ one has $\dl(\cdot\,,\dom)=d_{l^1}(\,\cdot\,,H_j)$,
and thus in this neighborhood $\dl(\cdot\,,\dom)$ is regular and satisfies the system of eikonal
equations (\ref{pbdipart}).

{\em Case 2.\,} Assume that $\displaystyle \dl(x,\dom) = \dl(x,F_j) = \dl(x,F_k)$
is attained in $F_j \cap F_k$ with $j \neq k$.
Then $\dl(x,F_j) = \dl(x,F_k)$ is attained in the intersection $H_j \cap H_k$.
We notice that the set of the elements $y$ of $\R^N$ such that $\dl(x,F_j) = \dl(x,F_k)$
is attained in $F_j \cap F_k$ is included in the hyperplane $A_{j,k}$
generated by $H_j \cap H_k$ and $\xi_j+\xi_k$, where $\xi_l$ denotes a normal vector
to $H_l$ for $l \in \{j,k\}$.
Since $\mathcal{L}^N(A_{j,k})=0$
this set is negligeble in view of the system of eikonal equations (\ref{pbdipart}).

{\em Case 3.\,} Assume that $\displaystyle \dl(x,\dom) = \dl(x,F_j) = \dl(x,F_k)$
for $j \neq k$ such that the hyperplanes $H_j$ and $H_k$ are parallel
and that case 2 is excluded. Then $\dl(x,F_l)$ is attained in the relative interior of $F_l$
for $l \in \{j,k\}$, and thus one has $\dl(x,H_j) = \dl(x,H_k)$. As a consequence,
$x$ belongs to the set $B_{j,k}:=\{y : \dl(y,H_j) = \dl(y,H_k)\}$, which is an hyperplane
parallel to $H_j$ and $H_k$, so that $\mathcal{L}^N(B_{j,k})=0$. This means that
this set is negligeble in view of the system of eikonal equations (\ref{pbdipart}).

As a consequence of the above discussion, almost every $x$ in $\Omega$ enters
case $1$, so that $d_{l^1}(\,\cdot\,,\dom)$ is a solution of (\ref{pbdipart}).
Moreover, if $i \in \{1,\ldots,N\}$ then one has
\begin{eqnarray*}
J^{d_{l^1}(\,\cdot\,,\dom)}_i & = &
J^{d_{l^1}(\,\cdot\,,\dom)}_i \cap \bigcup_{j,k} \left(A_{j,k} \cup B_{j,k} \right)
 \\
& \subset &
\bigcup_{j,k} \left[\om \cap \left(A_{j,k} \cup B_{j,k} \right)\right].
\end{eqnarray*}
Since $\om$ is bounded and the sets $A_{j,k}$ and $B_{j,k}$ are hyperplanes of $\R^N$,
we infer that the last union has finite $\Haus^{N-1}$ measure.
This being true for all index $i$, we get that $\F(d_{l^1}(\,\cdot\,,\dom))$ is finite.
\end{proof}

We can now turn to the proof of Theorem \ref{existence_particular_case}.
\begin{proof}[Proof of Theorem \ref{existence_particular_case}]
As a direct consequence of Lemma \ref{distanceH}, we get that the infimum
$\inf(P_1)$ is finite.
We now apply the direct methods of the Calculus of Variations
to prove the existence of an optimal solution of $(P_1)$.  

Let $(v^n)_{n \in \N} \subset \mathcal{E}_{pw}(\Omega)$ be a minimizing sequence for $(P_1)$.
We deduce from Lemma \ref{Ebound} that $(v^n)_n$ is uniformly bounded in $L^{\infty}(\om)$.
Moreover this sequence is uniformly Lipschitz continuous on $\barom$ since
$\dom$ is Lipschitz and $| \nabla v^n | \leq \sqrt{N}$ a.e. in $\om$ for any $n$.
Applying Ascoli-Arzel\`a and Banach-Alaoglu Theorems,
and up to a subsequence,
we can assume that $v^n \to v^\infty$ in ${C}_0(\Omega)$
and $v^n \to v^\infty$ weakly* in $W^{1,\infty}(\Omega)$ for some
$v^\infty\in W_0^{1,\infty}(\Omega)$.

\noindent
{\it Step 1}:
we show that $v^{\infty} \in \mathcal{E}_{pw}(\Omega)$.
We first prove that $\left|\frac{\partial v^\infty}{\partial x_i}\right|=1$ a.e. in $\om$.
Let us fix $i \in \{1,\ldots,N\}$.
We observe that since $(v^n)_n$ is a minimizing sequence, there exists a constant $C>0$ such that
$\F(v^n)\leq C$ for every $n \in \N$.
This implies that
\begin{equation}\label{vnbound}
\norm{\frac{\partial v^n}{\partial x_i}}_{BV(\om)} \,\,\, = \,\,\,
\Leb^N(\om) +
2\mathcal{H}^{N-1}(\Saut_i^{v^n}) \,\,\, \leq \,\,\,\Leb^N(\om) + 2 C .
\end{equation}
Applying Theorem \ref{compactnessBV} to the sequence
$\frac{\partial v^n}{\partial x_i}$ we have that
 $\frac{\partial v^n}{\partial x_i} \to g_i$ weakly*
in $BV(\om)$ for some
function $g_i \in BV(\om)$ (up to a subsequence).
Since we already know that
$\frac{\partial v^n}{\partial x_i} \to \frac{\partial v^\infty}{\partial x_i}$
weakly* in $L^\infty(\om)$,
we infer that $\frac{\partial v^\infty}{\partial x_i}=g_i$.
Moreover $\frac{\partial v^n}{\partial x_i} \to \frac{\partial v^\infty}{\partial x_i}$
in $L^1(\om)$
so that $|\frac{\partial v^\infty}{\partial x_i}|=1$ a.e. in $ \om$.

To prove that $v^\infty$ belongs to $C^1_{pw}(\Omega)$
it is sufficient to apply Theorem \ref{piecewise_constant_compactness}
to the sequence $\frac{\partial v^n}{\partial x_i}$. Indeed,
$\frac{\partial v^n}{\partial x_i}$ is a piecewise constant function for every $n$ and
$\norm{\frac{\partial v^n}{\partial x_i}}_{L^{\infty}(\Omega)}+
\mathcal{H}^{N-1}(J_i^{v^n})$ is uniformly bounded since $\F(v^n)\leq C$.
As a consequence, up to a subsequence, $\frac{\partial v^n}{\partial x_i}$ converges in measure to a piecewise
constant function which is necessarily $\frac{\partial v^{\infty}}{\partial x_i}$.
This proves that $v^{\infty} \in \mathcal{E}_{pw}(\Omega)$.

\noindent{\it Step 2:}
we show that $v^\infty$ is a minimizer of $\F$. Since $(v_n)_n$ is a minimizing sequence
for $\F$ it is sufficient
to prove that
\begin{equation}\label{semicontinuity_particular_case}
\liminf_{n \to \infty} \, \mathcal{H}^{N-1}(\Saut_i^{v^n})
\,\, \geq \,\, \mathcal{H}^{N-1}(\Saut_i^{v^\infty})
\end{equation}
for all $i=1,\ldots,N$.
Thanks to (\ref{DVrepresentation}) one has
for any $n \in \N\cup \{+\infty\}$
$$
\displaystyle \left| D \frac{\partial v^n}{\partial x_i}\right|(\om) = 2 \mathcal{H}^{N-1}(\Saut_i^{v^n})\,.
$$
Moreover we know that $\frac{\partial v^n}{\partial x_i}\to
\frac{\partial v^{\infty}}{\partial x_i}$ weakly* in $BV(\Omega)$. Applying Theorem \ref{semicontinuityBV} with $f\equiv 1$ we get
(\ref{semicontinuity_particular_case}).
\end{proof}
We now study the relationship between
the viscosity solutions of problem (\ref{pbdipart}) and the optimal solutions of problem $(P_1)$.
We restrict ourselves to the one and two-dimensional cases, where the geometric arguments are
intuitive.

In the one-dimensional case, it is evident that the viscosity solution  
$u(x)=1-|x|$ and its opposite minimize the number of jumps of the solutions of (\ref{pbdim1}), as their derivatives have just one jump.
For the 2-dimensional case, we recall the following result about the existence of a viscosity solution
(see Example 1.3 and Theorem 3.3 of \cite{Giovanni}) .
\begin{prop}
Let $F:\R^2\to \R$ be a continuous function
such that $F(\xi_1,\xi_2)=0$ if and only if $|\xi_i|=1, i=1,2$ and $F(\xi_1,\xi_2)<0$ if
$|\xi_i|<1$ for $i=1,2$.  Let $\Omega$ be an open bounded connected subset of $\R^2$.
Then there is no $W^{1,\infty}(\Omega)$
viscosity solution of
problem
\[
\left\{
\begin{array}{ll}
F(Dv)=0\,\,
&\mbox{a.e. in} \,\Omega
\vspace{0.1cm}\\
v=0 & \mbox{on} \,\,\, \partial \Omega,
\end{array}
\right.
\]
unless $\Omega$ is a rectangle whose normals are in the set
$E=\{x=(x_1,x_2) \in \R^2: |x_i|=\frac{1}{\sqrt{2}}\,,i=1,2\}\,.$
In this case the only viscosity solution is
$\dl(\cdot\,,\dom)$.
\end{prop}
It also happens that $\dl(\cdot\,,\dom)$ and its opposite are the only optimal
solutions of $(P_1)$ if $\Omega$ is a rectangle with normals in
$E$, as  the following result shows. Notice that if $\Omega \subset \R^2$
is convex and satisfies
property $(H)$, then $\Omega$ is necessarily such a rectangle.
\begin{prop}\label{viscosity_distance}
Let $\Omega$ be a rectangle with normals in $E$.
Then $\pm \dl(\cdot\,,\dom)$ are the only optimal solutions of $(P_1)$.
\end{prop}
We will use the following lemma in the proof:
\begin{lemma}\label{federer}
If $P: \R^2\to \R$ is the projection $(x_1,x_2)\to x_2$,
and $E$ is a measurable set of $\R^2$ then
$$
\int_{\R}\mathcal{H}^0(E\cap P^{-1}\{y\})dy\leq \mathcal{H}^{1}(E).
$$
\end{lemma}
\begin{proof}
It is sufficient to apply Theorem 2.10.25 of \cite{Federer} with $f=P$, $X=\R^2$, $Y=\R$, $k=0$ and $m=1$.
\end{proof}
We are now able to  prove Proposition \ref{viscosity_distance}:
\begin{proof}
To simplify the notations, we set $d:=\dl(\cdot\,,\dom)$. Let $v
\in \mathcal{E}_{pw}(\Omega)$ different
from $d$ and $-d$: we want to prove that
$$
\F(v) \, = \, \sum_{i=1}^2 \Haus^1(\Saut_i^v) \, > \, \sum_{i=1}^2 \Haus^1(\Saut_i^d)
\, = \, \F(d) \, = \, \F(-d).
$$
We first claim that $\Haus^1(\Saut_1^v) > \Haus^1(\Saut_1^d)$.
To this end, we define
$$
A \, := \; \left\{ t \in \R \,: \, \mbox{$\exists\, s \in \R$ such that $x:=(s,t) \in \om$ and
$-d(x) < v(x) < d(x)$}\right\}.
$$
Since the functions $v$, $d$ and $-d$ are continuous, the set $A$ can be written
as $\displaystyle \bigcup_{i \in I} \,]t_i,t'_i[\,$
where $I \subset \N$ is non empty, and $t_i < t'_i$ for all $i \in I$.
We notice that
$$
\Haus^1(\Saut_1^v \cap (\R \times \overline{A}^c))
= \Haus^1(\Saut_1^d \cap (\R \times \overline{A}^c))
$$
where $\overline{A}^c = \R \setminus \overline{A}$. We also notice that
$\Haus^1(\Saut_i^d \cap (\R \times \partial A)) = 0$,
so that in fact
$$
\Haus^1(\Saut_1^v \cap (\R \times A^c))
\geq \Haus^1(\Saut_1^d \cap (\R \times A^c)).
$$
For $\mathcal{L}^1$-almost every $t \in A$,
function $w: s \mapsto v(s,t)$
is differentiable on $\,]a,b[\,$ with derivative $\pm 1$,
where $a$ and $b$ are such that $](a,t),(b,t)[ = \om \cap (\R \times \{t\})$.
Moreover, $-d((s,t)) < w(s) < d((s,t))$ for some $s \in \,]a,b[\,$, so that
$w'=\frac{\partial v}{\partial x_1}$ has at least two jumps in $ \,]a,b[\,$.
As a consequence, applying Lemma \ref{federer}, we get
\begin{eqnarray*}
\Haus^1(\Saut_1^v \cap (\R \times A)) & \geq &
\int_{t \in A} {\mathcal H}^0 (\Saut_1^v \cap (\R \times \{t\})) dt \\
& \geq & \int_{t \in A} 2 dt = 2\, \mathcal{L}^1(A).
\end{eqnarray*}
On the other hand, the geometry of $\Saut_1^d$ implies that
$$
\Haus^1(\Saut_1^d \cap (\R \times A)) \,\,\leq\,\, \sqrt{2} \,\mathcal{L}^1(A).
$$
As a consequence we get $\Haus^1(\Saut_1^v) > \Haus^1(\Saut_1^d)$. The
inequality $\Haus^1(\Saut_2^v) > \Haus^1(\Saut_2^d)$ follows in the same way.
\end{proof}
We observe that $\dl(\cdot\,,\dom)$ is not in general an optimal solution of
$(P_1)$, as the following two examples show.
\begin{example}\label{example_Thierry}
Consider the subsets
$\om_+ := \{x \in \R^2: \| x \|_{l^1} \leq 3 \}$ and
$\om_- := \{x \in \R^2: \| x - (2,2)\|_{l^1} \leq 1 \}$
and set $\om := int(\om_+ \cup \om_-)$.
Define $u$ as
$$
u(x)=\dl(x,\dom_+)\chi_{\Omega_+}(x)-\dl(x,\dom_-)\chi_{\Omega_-}(x)\,.
$$
Then $\F(u) = 16,$ while $\F(\dl(\cdot\,,\dom))=16+2\sqrt{2}$.
\end{example}

The fact that
the function $\dl(\cdot\,,\dom)$ is not a solution of problem $(P_1)$,
is not a question of sign, as one could think from the previous example (notice that $u$ is not positive on $\Omega$).
Indeed in the next example we will exhibit
a positive
function $v \in \mathcal{E}_{pw}(\Omega)$ such that $\F(\dl(\cdot\,,\dom))>\F(v).$

\begin{example}
Define the four following squares $C_1 := \{x \in \R^2: \| x \|_{l^1} \leq 12 \}$,
$C_2 := \{x \in \R^2: \| x -(10,10)\|_{l^1} \leq 8 \}$,
$C_3 := \{x \in \R^2: \| x -(2,11)\|_{l^1} \leq 1 \}$,
$C_4 := \{x \in \R^2: \| x -(11,2)\|_{l^1} \leq 1 \}$.
Let $\Omega:=int(C_1\cup C_2\cup C_3\cup C_4)$. It is easy to see that
$\F(\dl(\cdot\,,\dom)) =120+16\sqrt{2}$.
Now, let us define
$$
u(x)=d_{l^1}(x,\partial {(C_1\cup C_2)})\chi_{C_1\cup C_2}(x)+
d_{l^1}(x,\partial {C_3})\chi_{C_3}(x)+ d_{l^1}(x,\partial {C_4})\chi_{C_4}(x).
$$
Then $u$ is non-negative, $u$ is a solution of (\ref{pbdipart}) and
$\F(u)=88+24\sqrt{2}$,
and thus $\dl(\cdot\,,\dom)$ is not optimal for $(P_1)$ among the non-negative functions of
$\mathcal{E}_{pw}(\Omega)$.
\end{example}
\section{The general case}\label{sec-gen-case}
We now turn to the general case 
where $\Omega$ is any bounded open connected subset of $\R^N$ with Lipschitz boundary. 
In the previous section we saw that the functional 
$\F$ attains its minimum in $\E_{pw}(\om)$, if the set $\Omega$ satisfies assumption $(H)$. 
In the general case 
we can't expect  $\F$ to be finite in $\mathfrak{S}(\om),$ as the 
following example shows.
\begin{example}
Let $\Omega=(0,1)^2$: we claim that 
$\F(v)=+\infty$ for every $v \in \mathfrak{S}(\om).$
Indeed, let us consider, for $t\in (0,1/2)$ fixed, the points 
$(s,t)$ in $\Omega$ with $s \in (0,1)$. 
Then for almost every $t \in (0,1/2)$ the
function $w_t : s \mapsto v(s,t)$ has the following properties: 
$w_t'=\frac{\partial v}{\partial x_1}(\,\cdot\,,t),$ 
$|w_t'|=1$ a.e., $w_t(0)=w_t(1)=0$ and $|w_t|\leq t$, 
due to Lemma \ref{Ebound}.
This implies that $w_t'$ has at least $\left[\frac{1}{2t}\right]$ jumps,
if $[x]$ denote the integer part of a real number $x$. 
Therefore, applying Lemma \ref{federer}, we have
$$
\Haus^1(\Saut_1^v) \,\, \geq \,\, 
\int_0^{1/2}
\mathcal{H}^0(\Saut_1^v\cap((0,1)\times {t}))dt
\,\, \geq \,\,
 \int_0^{1/2}\left[\frac{1}{2t}\right]dt \,\, =\,\, +\infty\,,
$$
and so $\F(v)=\displaystyle \sum_{i=1}^2 \Haus^1(\Saut_i^v)=+\infty\,.$
\end{example}
In view of the last example, a possible way to generalize the analysis
of the previous section is to measure the jump set 
of $Du$ for $u  \in \mathfrak{S}(\Omega)$ through a ``weight'' 
function $h \in C_0(\Omega)$ which penalizes the jumps of $Du$ in the interior
of $\Omega$.

We propose the following construction for the function $h$.
Let $(\om_n)_{n \geq 1}$ be a sequence of open subsets of $\Omega$ satisfying hypothesis $(H)$
and approximating $\om$ in the following way:
$$
\forall n \geq 1  \quad \quad \frac{1}{n+1/2} \leq
d(\om_n,\dom) \leq \frac{1}{n}.
$$
We set $\om_0 = \emptyset$. 
Notice that the sequence $(\om_n)_{n \geq 1}$ is increasing,
and that for any $n \geq 1$ the open set $\omega_n := \om_n \setminus \overline{\om_{n-1}}$ satisfies
hypothesis $(H)$. Then for any $n \geq 1$, consider the problem 
\[
(P_1^n)\quad\quad\quad \inf\left\{\,\sum_{i=1}^N \Haus^{N-1}(\Saut_i^v)
\, : \, v \in \mathcal{E}_{pw}(\omega_n) \right\}.
\]
It follows from Theorem \ref{existence_particular_case} that $\displaystyle \delta_n:=\inf(P_1^n) < +\infty$
for any fixed $n \geq 1$.
We now define $h \in C_0(\om, \,]0,+\infty[\,)$ by
$$
h(x) := 
\left\{ 
\begin{array}{l}
\vspace{2mm}
1/[4(\alpha_1+\alpha_2)] \quad \quad \quad \mbox{if $x \in \om_1$,}\\
\vspace{2mm}
\displaystyle \frac{1}{d_{l^1}(x,\dom_{n})+d_{l^1}(x,\dom_{n+1})}
\left[ \frac{d_{l^1}(x,\dom_{n+1})}{(n+1)^2 (\alpha_n + \alpha_{n+1})}
+ \frac{d_{l^1}(x,\dom_{n})}{(n+2)^2 (\alpha_{n+1} + \alpha_{n+2})} \right] \\
\quad \quad \quad \quad \quad \quad \quad \quad \mbox{if $x \in \omega_{n+1}$, $n \geq 1$,}\\
\end{array} \right.
$$
where $\alpha_n := \max\{1, \delta_n + \mathcal{H}^{N-1}(\dom_n)\}$ for any $n \geq 1$. 

Define 
$$\mathcal{E}(\Omega):=
\{v \in \mathfrak{S}(\Omega): Dv \in SBV_{loc}(\Omega)^N\}\,.$$
We will show the following result:
\begin{thm}\label{existence}
Let $\Omega$ be an open bounded connected subset of $\R^N$ 
with Lipschitz boundary.
Let $h$ defined as above. Then the variational problem 
$$
(P_h) \quad \quad \quad 
\inf\left\{\,\F_{h}(v) \,:= \,\sum^N_{i=1}\int_{{J}_i^v}
h(x) d\mathcal{H}^{N-1}(x) \, : v \in \mathcal{E}(\Omega) \right\}.
$$
has finite value and has a optimal solution.
\end{thm}
\begin{rem}
As for problem $(P_1)$, $(P_h)$ has at least two optimal solutions: if $v$ is a solution, then $-v$ 
is a solution too.
\end{rem}
\begin{rem}
If $\Omega$ satisfies hypothesis $(H)$, one can choose $h\equiv 1$ as
Theorem \ref{existence_particular_case} shows. 
\end{rem}
\begin{proof}
We divide the proof into two steps: in the first one we show that $\inf(P_h)$ is finite
and in the second one we prove the
existence of an optimal solution.

\noindent
{\it Step 1:}
We show that there exists a function $u \in \mathcal{E}(\Omega)$ such that
$\mathcal{F}_{h}(u)$ is finite.
For any $n \geq 1$ let $\omega_n$ be the sets considered in the definition of $h$.
Let $u_n$ be an optimal solution of the associated problem $(P_1^n)$, extended
by $0$ on $\R^N \setminus \omega_n$.
(such a solution exists due to Theorem \ref{existence_particular_case}).  We then define $u \in \E(\om)$ by $u := \sum_{n=1}^{+\infty} u_n(x)$,
and claim that $\F_{h}(u) < +\infty$.
Indeed, using that
$h \leq \frac{2}{\alpha_n n^2}$ on $\overline{\omega}_n$ for any $n \geq 1$, one can estimate, for every $i$
\begin{eqnarray*}
\int\limits_{{J}_i^u}h(x)\,d\mathcal{H}^{N-1}(x)
& = &
\sum\limits_{n=1}^{+\infty} \,\,\,
\int\limits_{{J}_i^u \cap \left(\omega_{n} \cup \dom_n\right)} h(x)\,d\mathcal{H}^{N-1}(x) \\
& = & \sum\limits_{n=1}^{+\infty} \left[ \,\,
\int\limits_{{J}_i^{u_n}\cap \omega_n} h(x)\,d\mathcal{H}^{N-1}(x)
\, + \, \int\limits_{{J}_i^u\cap \partial \Omega_{n}}h(x)\,d\mathcal{H}^{N-1}(x) \right] \\
& \leq & \sum\limits_{n=1}^{+\infty} \frac{2}{n^2} \left[ \,\,
\int\limits_{{J}_i^{u_n}\cap \omega_n} \frac{1}{\delta_n} \,d\mathcal{H}^{N-1}(x)
\, + \,
\int\limits_{{J}_i^u\cap \partial \Omega_{n}} \frac{1}{\mathcal{H}^{N-1}(\dom_n)} \,d\mathcal{H}^{N-1}(x) \right] \\
& \leq & \sum\limits_{n=1}^{+\infty} \frac{4}{n^2} \,\, < \,\, +\infty .
\end{eqnarray*}
As a consequence $\inf(P_h)$  has finite value.

\noindent
{\it Step 2:}
We  show that there exists an optimal solution of problem $(P_h)$
applying the direct methods of the Calculus of Variations.
Let $(v^n)_n \subset \mathcal{E}(\Omega)$ be a minimizing
sequence for $(P_h)$. As in the proof of Theorem \ref{existence_particular_case}
we can assume that, up to a subsequence, $v^n\to v^{\infty}$ in $ {C}_0({\Omega})$ and
weakly* in $W^{1,\infty}(\Omega)$ for some $v^{\infty}\in W^{1,\infty}_0(\Omega)$.
It remains to prove that $v^{\infty}$ is an optimal solution of $(P_h).$

Let us prove that $v^{\infty} \in \mathcal{E}(\Omega).$
Since $(v^n)_n$ is a minimizing sequence we can assume that there exists $C>0$ such
that $\mathcal{F}_{h}(v^n)\leq C$ for every $n \in \N$.
Fix an open subset $\omega$ of $\om$ with Lipschitz boundary and such that
$\omega\subset\bar{\omega}\subset\om$, and let $\alpha>0$ be such
that $0<\alpha\leq h\leq 1$ in $\bar{\omega}$. We can then write
$$
C\geq \mathcal{F}_{h}(v^n)\geq
\alpha\sum^N_{i=1}\int\limits_{{J}_{i}^{v^n}\cap
\omega} d\mathcal{H}^{N-1}(x), \quad \forall \,n \in \N.
$$
Let us fix $i \in \{1,\ldots,N\}$.
Since $\frac{\partial v^n}{\partial x_i}
\in SBV(\omega)$ and takes only the two values $\pm 1$ for any $n$, we can then write
$$
\norm{\frac{\partial v^n}{\partial x_i}}_{BV(\omega)}\,\,=\,\,
\mathcal{L}^N(\omega)+2\mathcal{H}^{N-1}({J}_i^{v^n} \cap \omega)\,\, \leq \,\, \mathcal{L}^N(\omega)+2 C.
$$
We now apply Theorem \ref{compactnessSBV} to the sequence
$(\frac{\partial v^n}{\partial x_i})_n$:
condition $i)$ is the inequality given above;
condition $ii)$ is straightforward since $\nabla \frac{\partial v^n}{\partial x_i} = 0$ a.e. in $\omega$;
for condition $iii)$ it is sufficient to choose $f \equiv 1$, for example,
because in this way
$$
\forall n \in \N, \quad \quad 
\int_{\Saut^{v^n}_i \cap \omega} f \left( \left[\frac{\partial v^n}{\partial x_i}\right] \right) d\mathcal{H}^{N-1}(x)
\,\,\, = \,\,\, \mathcal{H}^{N-1}(\Saut^{v^n}_i \cap \omega) \,\,\, \leq \,\,\, C.
$$
As a consequence $\frac{\partial v^n}{\partial x_i} \to g_i$ weakly*
in $BV(\omega)$ for some
function $g_i \in SBV(\omega)$ .
Since we already know that
$\frac{\partial v^n}{\partial x_i} \to \frac{\partial v^\infty}{\partial x_i}$
weakly* in $L^\infty(\om)$,
we infer that $\frac{\partial v^\infty}{\partial x_i} \in SBV(\omega)$.
Moreover $\frac{\partial v^n}{\partial x_i} \to \frac{\partial v^\infty}{\partial x_i}$
in $L^1(\omega)$
so that $|\frac{\partial v^\infty}{\partial x_i}|=1$ a.e. in $\omega$.
This being true for any index $i$ and any such open subset $\omega$ of $\Omega$,
we infer that $v^\infty$ belongs to $\E(\Omega)$.

To prove that $v^{\infty}$ is a minimizer of $\F_h$,
let us show that
\begin{equation}
\liminf\limits_{n\to \infty} \int\limits_{{J}_i^{v^n}\cap \omega} h(x)\,
d\mathcal{H}^{N-1}(x)\geq
\int\limits_{{J}_i^{v^{\infty}}\cap \omega}
h(x)\, d\mathcal{H}^{N-1}(x)\,\,\,i=1,\ldots,N
\label{fin}
\end{equation}
for any $\omega$ as above. To this aim we observe that for every $n\in \N \cup \{\infty\}$ one has
$$
\int\limits_{{J}_i^{v^{n}}\cap \omega} h(x) d\mathcal{H}^{N-1}(x)
\,\, = \,\, \frac{1}{2} \int\limits_{\omega}h(x)\, d\left|D\frac{\partial v^{n}}{\partial x_i}\right|(x).
$$
Therefore (\ref{fin}) follows from  Theorem \ref{semicontinuityBV}.
This is sufficient to conclude the proof, since
\begin{eqnarray*}
\liminf\limits_{n\to \infty} \int\limits_{{J}_i^{v^n}} h(x)\, d\mathcal{H}^{N-1}(x)
& \geq &
\sup_{\omega\subset \Omega}
\int\limits_{{J}_i^{v^{\infty}}\cap \omega} h(x)\, d\mathcal{H}^{N-1}(x)
\,\, = \,\, \int_{{J}_i^{v^{\infty}}} h(x)\, d\mathcal{H}^{N-1}(x)
\end{eqnarray*}
where the supremum is taken over the open subsets $\omega \subset \overline{\omega} \subset \Omega$
with Lipschitz boundary. We then infer from the last inequality that
$\displaystyle \liminf\limits_{n\to \infty} \mathcal{F}_{h}(v^n) \geq \mathcal{F}_{h}(v^\infty)$,
that is, $v^{\infty}$ is a minimizer of $\mathcal{F}_{h}$.
\end{proof}


\section{Open problems}\label{sec_open_problems}
We would like to present here some open problems related to our work: 
some questions about problems $(P_1)$ and $(P_h)$ and 
the question of the selection of a particular solution of the
differential  problem (\ref{pbdipart}).

\noindent
{\textbf 1.} Suppose that $\Omega$ satisfies property $(H).$ 
We showed in section \ref{sec-part-case} that there exists a function $u \in C^1_{pw}(\Omega)$ solution of problem $(P_1).$ 
Can we say that the number of the connected components of $Du$ is finite?
This seems to be natural, as shown in Proposition \ref{distanceH}. 

Another natural question is: has problem $(P_1)$ a finite number of solutions?
Does it admit just two optimal solutions (a solution and its opposite)? 
This is  true  when $\Omega$ is a rectangle and $N=2$
as shown in Proposition  \ref{viscosity_distance}. 
With the same arguments as in the proof of this proposition, 
it is easy to show that the functions $\pm u$, where $u$ is the solution defined
in Example \ref{example_Thierry}, are the only optimal solutions of problem $(P_1)$.

Another related question is to extend the result of Proposition  \ref{viscosity_distance}
to higher dimension, that is to prove that if $\Omega$ a convex open subset of $\R^N$ satisfying $(H)$
then the distance functions $\pm d_{l^1}(.,\dom)$ are the only solutions of problem $(P_1)$.
We notice that under the above hypothesis on $\Omega$, 
the function $d_{l^1}(.,\dom)$ is a concave solution of (\ref{pbdipart}): 
this  implies
 that on a non-void interval $\R \times {y} \cap \Omega$ (where $y \in \R^{N-1}$)
the partial derivative $\frac{\partial}{\partial x_1} d_{l^1}(.,\dom)$ has only one jump.
As a consequence the jump set of $\frac{\partial}{\partial x_1}d_{l^1}(.,\dom)$
is also in some sense minimal since on each such interval it is reduced to one point, and
so it has minimal cardinality.

\vspace{0.3cm}
\noindent{\textbf 2.}
In section \ref{sec-gen-case} we studied the general case,
that is when $\Omega$ is any subset of $\R^N$.
A natural question  is the study of the
asymptotic behavior of a family $(u_h)_h$ of the optimal solutions of $(P_h)$
when $h \to 1$ locally uniformly in $\Omega$. The limit problem is
then obviously $(P_1)$, which in general is ill-posed (since it has no
feasible solution), but one expects $u_h$ to converge uniformly 
to a function 
$u \in \E(\Omega)$ having some local minimization properties in the open
subsets $\omega\subset \overline{\omega}\subset\Omega$.

\vspace{0.3cm}
\noindent
{\textbf 3.} In section  \ref{sec-gen-case} we saw that for a general open set $\Omega$,
$\mathcal{F}(v)$ may be infinite for any $v \in \mathfrak{S}(\Omega)$.
This means that for some index $i$ one has $\mathcal{H}^{N-1}(J_i^v)=\infty$. 
On the other hand we know that $\mathcal{H}^{N}(J_i^v)=0$, so a natural question
is to identify the Hausdorff dimension $s\in [N-1,N]$ for which $\mathcal{H}^{t}(J_i^v)$
is finite if $t<s$, infinite if $t>s$. Then one may try to minimize this dimension $s$
over $\mathfrak{S}(\Omega)$, and then proceed as in section \ref{sec-part-case}
with the optimal $s$ in place of $N-1$.

\vspace{0.3cm}
\noindent
{\textbf 4.} In this article we proposed to minimize the measure of the discontinuity 
sets of the gradients of the elements of $\mathfrak{S}(\Omega)$ through the functionals 
$\mathcal{F}$ and $\mathcal{F}_{h}.$ 
Now, let $\{G_i\}_{i \in I}$  be the connected components of the $Du^{-1}(E)$
for $u \in \mathfrak{S}(\Omega)$, where we recall that $E=\{x:\forall i, \,\,|x_i|=\frac{1}{\sqrt{N}}\}$.
One could instead consider functionals of the form 
\[
\sum_{i\in I}g(\mathcal{H}^{N-1}(G_i))\quad with \quad
g(t)=
\left\{
\begin{array}{ll}
t^\beta, & t\leq 1
\\
t, & t\geq 1
\end{array}
\right. 
\] 
where $\beta \in ]0,1[\,$. 
The advantage of this last type of functional with respect to our choice 
is that
 we can penalize the number of components $G_i$ thanks to $g$.  

\vspace{0.3cm}
\noindent
{\textbf 5.}  One of the aim of this work was to select a particular
solution of $(\ref{pbdipart})$.
The idea developed in the previous section is first to restrict ourselves to those
solutions of (\ref{pbdipart}) that minimize problem $(P_h)$.
Now how can we select a particular solution of $(P_h)$ (up to its sign)? 
Does problem $(P_h)$ admit more than two solutions? 
Notice that in any case the selection of these two solutions of our problem
$(\ref{pbdipart})$ depends on the choice of $h$, so that it would be even
more interesting to select a solution obtained by letting $h$ go to $1$
(see {\bf 4.} above).

In the case where problem $(P_h)$ admits more than two solutions
one can think of the following way to select a particular solution.
Let $\mathcal{S}_0$ be the set of solutions of problem $(P_h)$,
and let $(f_n)_{n \geq 1}$ be an orthogonal basis of $L^2(\Omega)$ 
starting from $f_1 =1$.
With the same arguments as in the proof of Theorem \ref{existence},
it is not difficult to show that problem
$$
(B_1)\quad\quad
\sup\left\{\int_{\Omega}u f_1, u \in \mathcal{S}_0\right\}
$$
admits a solution. 
Let $\mathcal{S}_1$ be the set of optimal solutions of problem $(B_1)$. 
If $\mathcal{S}_1$ is composed by more then one function, we consider
$$
(B_2)\quad\quad
\sup\left\{\int_{\Omega}u f_2, u \in \mathcal{S}_1\right\}
$$
and so on. If two functions $u, v \in \mathcal{S}_0$ satisfy
$$
\int_{\Omega}u f_i =\int_{\Omega}v f_i  \quad \quad \forall i\geq 1
$$
then one gets $u=v$. Therefore this method allows us to select 
a particular solution of $(P_h)$.

\section{Appendix}
We recall here some compactness and semicontinuity results about 
$BV$,  
$SBV$ and piecewise constant functions; for the proofs, 
see respectively Theorem 3.23 in \cite{ambrosio_fusco_pallara}, 
Theorem 1.4 in \cite{almant}, Theorem 4.25 and Theorem 2.38 in \cite{ambrosio_fusco_pallara}. 
\begin{thm}[compactness for $BV$ functions]\label{compactnessBV}
Let $\Omega$ be a bounded open subset of $\R^N$ with Lipschitz boundary.
Assume that $(u_n)_n$ is an  uniformly bounded 
sequence in $BV(\Omega)$. Then 
there exist a subsequence $(u_{n_k})_k$ and 
a function $u \in BV(\Omega)$ such that
$u_{n_k} \to u$ weakly*
in $BV(\Omega)$.
\end{thm}

\begin{thm}[compactness for $SBV$ functions]\label{compactnessSBV}
Let $\Omega$ be a bounded open subset of $\R^N$ with Lipschitz boundary.
Let $(u_n)_n$ be a sequence of functions in $SBV(\Omega)$ and assume that
\\
$i)$ the functions $u_n$ are uniformly bounded in $BV(\Omega)$;
\\
$ii)$ the gradients $\nabla u_n$ are equi-integrable;
\\
$iii)$ there exists a function $f:[0,\infty)\to [0,\infty]$ such that
$f(t)/t \to \infty$ as $t \to 0^+$ and
$$
 \int_{\Saut_{u_n}}f([u_n])d\mathcal{H}^{N-1}
\leq C<\infty\quad \forall n \in \N.
$$
Then there exists a subsequence $(u_{n_k})_k$ and a function $u \in SBV(\Omega)$ such that
$u_{n_k}\to u$ weakly*  in $BV(\Omega)$,
with the Lebesgue and jump parts of the
derivatives converging separately, i.e. $D_a u_{n_k}\to D_a u$ and $D_j u_{n_k}\to D_j u$
weakly* in $\mathcal{M}(\om,\R^N)$.
\end{thm}

\begin{thm}[compactness for piecewise constant functions]
\label{piecewise_constant_compactness}
Let $\Omega$ be a bounded open subset of $\R^N$ with Lipschitz boundary.
Let $(u_n)_{n}$ be a 
sequence of piecewise constant functions in $\Omega$ such that 
$\norm{u_n}_{L^{\infty}(\Omega)}+\mathcal{H}^{N-1}(J_{u_n})$ is uniformly
bounded. 
Then there exist a subsequence $(u_{n_k})_k$ 
and a piecewise constant function $u$ such that $u_{n_k} \to u$ 
in measure.
\end{thm}

\begin{thm}[semicontinuity in $BV$]\label{semicontinuityBV}
Let $\Omega$ be a bounded open subset of $\R^N$.
Let $(u_n)_n$ be a sequence of functions in $BV(\Omega)$ such that $u_n \to u$ 
weakly* in $BV(\Omega).$ Then
$$
\int_\om f(x) d|D_j u|(x) \leq \liminf_{n \to \infty}\int_\om f(x) d|D_j u_{n}|(x)
$$
for any non-negative continuous function $f: \om \to [0,+\infty[\,$.
\end{thm}
\vspace{0.3cm}

\textit{Acknowledgments.}
The authors wish to express their thanks to Gr\'egoire Charlot,  
Bernard Dacorogna 
and J\'er\^ome Droniou for some stimulating conversations.

\bibliographystyle{plain}
\bibliography{eikonal}

\end{document}